\documentclass[a4paper,12pt]{amsart}
\usepackage{amsmath,amssymb,amsthm}
\usepackage{amscd}
\usepackage{array,enumerate}
\newtheorem{Thm}{Theorem}[section]
\newtheorem{Prop}[Thm]{Proposition}
\newtheorem{Lem}[Thm]{Lemma}
\newtheorem{Cor}[Thm]{Corollary}
\newtheorem{Thmint}{Theorem}[section]
\newtheorem{Corint}[Thmint]{Corollary}

\theoremstyle{definition}
\newtheorem{Rem}[Thm]{Remark}

\newtheorem{Def}[Thm]{Definition}

\newcommand{\Cs}{\mbox{${\rm C}^\ast$}}
\newcommand{\id}{\mbox{\rm id}}

\DeclareMathOperator{\bigfp}{\mbox{\lower0.50ex\hbox{\LARGE $\ast$}}}
\title[Ambient nuclear \Cs -algebras]{Group \Cs -algebras as decreasing intersection of nuclear \Cs -algebras}
\author{Yuhei Suzuki}
\address{Department of Mathematical Sciences,
University of Tokyo, Komaba, Tokyo, 153-8914, Japan}
\address{Research Institute for Mathematical Sciences, Kyoto University, Kyoto, 606-8502, Japan}
\email{suzukiyu@ms.u-tokyo.ac.jp}
\begin{document}
\begin{abstract}
We prove that for every exact discrete group $\Gamma$,
there is an intermediate \Cs -algebra between the reduced group \Cs -algebra and
the intersection of the group von Neumann algebra and the uniform Roe algebra which is realized as the intersection of a decreasing sequence of isomorphs of the Cuntz algebra $\mathcal{O}_2$.
In particular, when $\Gamma$ has the AP (approximation property),
the reduced group \Cs -algebra is realized in this way.
We also study extensions of the reduced free group \Cs -algebras and show
that any exact absorbing or unital absorbing extension of it by any stable separable nuclear \Cs -algebra
is realized in this way.
\end{abstract} 
\maketitle

\section{Introduction}
It is well-known that every exact discrete group admits an amenable action on a compact space \cite{Oz2},
and each such action gives rise to an ambient nuclear \Cs -algebra of the reduced group \Cs -algebra
via the crossed product construction \cite{Ana}.
More generally, it is known that every separable exact \Cs -algebra is embeddable into the Cuntz algebra $\mathcal{O}_2$ \cite{KP}.
Motivated by these phenomena, we are interested in the following question.
How small can we take an ambient nuclear \Cs -algebra/ Cuntz algebra $\mathcal{O}_2$ for a given exact \Cs -algebra?
In the present paper, we give an answer to the question for the reduced group \Cs -algebras of discrete groups with the AP.
The next theorem states that an ambient nuclear \Cs -algebra of the reduced group \Cs -algebras with the AP
can be arbitrarily small in a certain sense.
\begin{Thmint}\label{Thmint:Main}
Let $\Gamma$ be a countable discrete exact group.
Then there is an intermediate \Cs -algebra $A$ between the reduced group ${\rm C}^\ast$-algebra ${\rm C}_r^\ast(\Gamma)$ and the intersection of the group von Neumann algebra $L(\Gamma)$ and
the uniform Roe algebra ${\rm C}^\ast_u(\Gamma)$ satisfying the following properties.
\begin{itemize}
\item
There is a decreasing sequence of isomorphs of the Cuntz algebra $\mathcal{O}_2$ whose intersection is isomorphic to $A$.
\item
There is a decreasing sequence $(A_n)_{n=1}^\infty$ of separable nuclear \Cs -algebras whose intersection
is isomorphic to $A$ and the sequence admits compatible multiplicative conditional expectations
$(E_n\colon A_1\rightarrow A_n)_{n=1}^\infty$.
Here the compatibility means that the equality $E_n\circ E_m=E_n$ holds for all $n\geq m$. 
\end{itemize}
In particular, when the group $\Gamma$ has the AP,
the statements hold for the reduced group \Cs -algebra ${\rm C}_{r}^\ast(\Gamma)$.
\end{Thmint}

As a consequence of Theorem \ref{Thmint:Main}, we obtain the following result.
\begin{Corint}\label{Corint:Main}
The decreasing intersection of nuclear \Cs -algebras need not have the following properties.
\begin{enumerate}[\upshape (1)]
\item The OAP, hence nuclearity, the CBAP, the WEP, and the SOAP.
\item The local lifting property.
\end{enumerate}
They can happen simultaneously.
The statements are true even when the decreasing sequence admits a compatible family of multiplicative conditional expectations.
\end{Corint}
For the implication stated in (1), recall that for exact \Cs -algebras,
nuclearity is equivalent to the WEP \cite[Exercise 2.3.14]{BO}.
Thus the decreasing intersection of nuclear \Cs -algebras can lost most of good properties.
Since the decreasing intersection of injective von Neumann algebras
is injective, the analogous results for von Neumann algebras can never be true.

We also give a geometric construction of a decreasing sequence of Kirchberg algebras whose intersection
is isomorphic to the hyperbolic group \Cs -algebra.
Although the result follows from Theorem \ref{Thmint:Main},
this approach has good points. Our decreasing sequence is taken inside the boundary algebra $C(\partial \Gamma)\rtimes \Gamma$. Moreover, the proof does not depend on Kirchberg--Phillips's $\mathcal{O}_2$-absorption theorem and the theory of reduced free products, both of which are used in the proof of Theorem \ref{Thmint:Main}.
Using the sequence constructed by this method, we also study absorbing extensions of the reduced free group \Cs -algebra by
a stable separable nuclear \Cs -algebras,
and prove the following theorem.
\begin{Thmint}\label{Thmint:NE}
Let $A$ be a stable separable nuclear \Cs -algebra and let
$$0\rightarrow A \rightarrow B\rightarrow {\rm C}_{r}^\ast(\mathbb{F}_d)\rightarrow 0$$ be an extension of ${\rm C}_{r}^\ast(\mathbb{F}_d)$ by $A$ $(2\leq d \leq \infty)$.
Assume $B$ is exact and the extension is either absorbing or unital absorbing.
Then $B$ is realized as a decreasing intersection
of isomorphs of the Cuntz algebra $\mathcal{O}_2$.
In particular, any exact extension of ${\rm C}_{r}^\ast(\mathbb{F}_d)$ by $\mathbb{K}$
is realized in this way.
\end{Thmint}
The proof of Theorem \ref{Thmint:NE} is based on the KK-theory.
\subsection*{Organization of the paper}
In Section \ref{Sec:Pre},
we review some notions and facts used in the paper.
In Section \ref{Sec:proof}, we prove Theorem \ref{Thmint:Main}.
We also give few more examples satisfying the conditions in Theorem \ref{Thmint:Main}.
In Section \ref{Sec:gen}, we deal with the hyperbolic groups.
Based on the study of the boundary action,
we construct a decreasing sequence of nuclear \Cs -algebras
inside the boundary algebra $C(\partial\Gamma)\rtimes \Gamma$
whose intersection is the reduced group \Cs -algebra ${\rm C}^\ast_{r}(\Gamma)$.
In Section \ref{Sec:NE}, using the decreasing sequence constructed in Section \ref{Sec:gen},
we prove Theorem \ref{Thmint:NE}.
In Section \ref{Sec:app}, we study some amenable dynamical systems of the free groups constructed in Section \ref{Sec:gen}.
\subsection*{Notation}
The symbol `$\otimes$' stands for the minimal tensor product.
The symbol `$\rtimes$' stands for the reduced crossed product of \Cs -algebras.
For a discrete group $\Gamma$ and $g\in \Gamma$,
denote by $\lambda_g$ the unitary element of the reduced group \Cs -algebra ${\rm C}_{r}^\ast(\Gamma)$ corresponding to $g$.
For a unital $\Gamma$-\Cs-algebra $A$ and $g\in \Gamma$,
denote by $u_g$ the canonical implementing unitary element of $g$ in
the reduced crossed product $A\rtimes \Gamma$.
With the same setting, for $x\in A\rtimes \Gamma$ and $g\in \Gamma$,
the $g$th coefficient $E(xu_g^{\ast})$ of $x$ is denoted by
$E_g(x)$. Here $E\colon A\rtimes\Gamma\rightarrow A$ denotes the canonical conditional expectation of the reduced crossed product.
We denote by $\mathbb{K}$ and $\mathbb{B}$
the \Cs -algebras of all compact operators and all bounded operators on $\ell^2(\mathbb{N})$ respectively.
For a set $X$, denote by $\Delta_X$
the diagonal set $\{(x, x):x\in X\}$ of $X\times X$.
For a subset $Y$ of a topological space $X$,
denote by ${\rm int}(Y)$ and ${\rm cl}(Y)$
the interior and the closure of $Y$ in $X$ respectively.

\section{Preliminaries}\label{Sec:Pre}
\subsection{Amenability of group actions on compact spaces and \Cs-algebras}
Recall that an action $\Gamma\curvearrowright X$ of a group $\Gamma$ on a compact Hausdorff space
is said to be amenable if there is a
net $(\zeta_i\colon X \rightarrow {\rm Prob}(\Gamma))_{i\in I}$
of continuous maps satisfying 
\[\lim_{i\in I}\left( \sup_{x\in X}\|g.\zeta_i(x)-\zeta_i(g.x)\|_1 \right)= 0 {\rm\ for\ all\ }g\in \Gamma.\]
Here ${\rm Prob}(\Gamma)$ is the space of probability measures on $\Gamma$ equipped with the pointwise convergence topology.
More generally, an action of $\Gamma$ on a unital \Cs -algebra $A$
is said to be amenable if the induced action of $\Gamma$ on the spectrum of the center $Z(A)$ of $A$ is amenable.
Here we only review a few properties of amenability of group actions.
We refer the reader to \cite[Section 4.3]{BO} for the details.
\begin{Prop}\label{Prop:ame}
Let $\Gamma\curvearrowright A$ be an amenable action of a group $\Gamma$
on a unital \Cs -algebra $A$.
Then the following hold.
\begin{itemize}
\item The full and the reduced crossed products coincide.
\item The $($reduced$)$ crossed product $A\rtimes \Gamma$ is nuclear if and only if $A$ is nuclear.
\end{itemize}
\end{Prop}
\subsection{Approximation properties for \Cs -algebras and groups}
For \Cs -algebras $A$, $B$ and a closed subspace $X$ of $B$,
we define a subspace $F(A, B, X)\subset A\otimes B$
by
$$F(A, B, X):=\{x\in A\otimes B: (\varphi\otimes \id_B)(x)\in X{\rm\ for\ all\ }\varphi\in A^\ast\}.$$
A triplet $(A, B, X)$ is said to have the slice map property if the equality
$F(A, B, X)=A\otimes X$ holds.
Here $A\otimes X$ denotes the closed subspace of $A\otimes B$ spanned by elements of the form
$a\otimes x; a\in A, x\in X$.
We give a definition of the SOAP (strong operator approximation property) and the OAP (operator approximation property) in terms of the slice map property. See \cite[Section 12.4]{BO} for the detail.
\begin{Def}
A \Cs -algebra $A$ is said to have the SOAP (resp.\ the OAP) if for any \Cs -algebra $B$ (resp.\ for $B=\mathbb{K}$)
and for any closed subspace $X$ of $B$,
the triplet $(A, B, X)$ has the slice map property.
\end{Def}
Obviously, we have the implications
\begin{center}
Nuclearity $\Rightarrow$ CBAP $\Rightarrow$ SOAP$\Rightarrow$ OAP, Exactness.
\end{center}
All implications are known to be proper.
However, for the reduced group \Cs -algebras, the SOAP and the OAP are equivalent.
The SOAP and the OAP have the strong connection with the property of groups called the AP (approximation property).
Here we give the following equivalent condition as a definition of the AP.
\begin{Def}\label{Def:AP}
A discrete group $\Gamma$ is said to have the AP if there exists a net $(\varphi_i)_{i\in I}$
of finitely supported complex valued functions on $\Gamma$ such that
$m_{\varphi_i}\otimes \id_{\mathbb{B}}$ converges to the identity map
in the pointwise norm topology.
Here, for a finitely supported function $\varphi$ on $\Gamma$,
denote by $m_\varphi\colon {\rm C}^\ast_{r}(\Gamma) \rightarrow {\rm C}^\ast_{r}(\Gamma)$
the completely bounded map defined by the formula
$m_\varphi(\lambda_g):= \varphi(g)\lambda_g$ for $g\in \Gamma$.
\end{Def}
This property is characterized in the following way.
\begin{Prop}\label{Prop:AP}
Let $\Gamma$ be a discrete group.
Then the following are equivalent.
\begin{enumerate}[\upshape (1)]
\item The group $\Gamma$ has the AP.
\item The \Cs -algebra ${\rm C}^\ast_{r}(\Gamma)$ has the SOAP.
\item The \Cs -algebra  ${\rm C}^\ast_{r}(\Gamma)$ has the OAP.
\item There is an intermediate \Cs -algebra between ${\rm C}^\ast_{r}(\Gamma)$ and $L(\Gamma)$ which has the SOAP or the OAP.

\end{enumerate}

\end{Prop}
See \cite[Section 12.4]{BO} for the proof.
Note that the implication (4)$\Rightarrow$(1) follows from the proof of (2), (3)$\Rightarrow$ (1).

Next recall that the group von Neumann algebra of a group $\Gamma$
is the von Neumann algebra $L(\Gamma)$ generated by the image of the left regular representation of $\Gamma$.
We also recall that the uniform Roe algebra of $\Gamma$ is the \Cs -algebra ${\rm C}^\ast_u(\Gamma)$
on $\ell^2(\Gamma)$ generated by $\ell^\infty(\Gamma)$ and the reduced group \Cs -algebra of $\Gamma$.
Note that both algebras contain the reduced group \Cs -algebra by definition.
A group $\Gamma$ is said to have the ITAP (invariant translation approximation property \cite{Roe})
if we have the equality
$$L(\Gamma)\cap {\rm C}^\ast_u(\Gamma)= {\rm C}^\ast_{r}(\Gamma).$$
Here the above equality makes sense by regarding all algebras as \Cs -algebras on $\ell^2(\Gamma)$
in the canonical way.
We remark that under the canonical isomorphism
${\rm C}^{\ast}_u(\Gamma)\cong \ell^{\infty}(\Gamma)\rtimes \Gamma$ (cf. Proposition 5.1.3 in \cite{BO}), the intersection
$L(\Gamma)\cap {\rm C}^\ast_u(\Gamma)$ is identified with
the \Cs -subalgebra of $\ell^{\infty}(\Gamma)\rtimes \Gamma$ consisting of elements whose coefficients sit in $\mathbb{C}$.
It is shown by Zacharias \cite{Zac} that every group with the AP has the ITAP.
See also Proposition \ref{Prop:AP2}. We do not know either the ITAP holds or not for groups without the AP.

\subsection{Reduced free product}
We refer the reader to \cite[Section 4.7]{BO} for the definition of the reduced free product.
First we recall a few terminology related to theorems we will use.
Let $A$ be a \Cs -algebra and $\varphi$ be a state on $A$.
Recall that $\varphi$ is said to be non-degenerate if its GNS-representation is faithful.
Recall that the centralizer of $\varphi$ is
the set of all elements $a\in A$ satisfying the equality
$\varphi(ab)=\varphi(ba)$ for all $b\in A$.
An abelian \Cs -subalgebra $D$ of $A$ is said to be diffuse with respect to $\varphi$
if $\varphi|_D$ is a diffuse measure on the spectrum of $D$.

In the proofs of Theorems \ref{Thmint:Main} and \ref{Thmint:NE}, we use the reduced free product to make \Cs -algebras simple.
The following two theorems are important in our proof.
The first theorem guarantees the nuclearity of the reduced free product under certain conditions.
The second one gives a sufficient condition for the simplicity of the reduced free product.
\begin{Thm}[Dykema--Smith {\cite[Exercise 4.8.2]{BO}}]\label{Thm:DS}
Let $(A, \varphi)$ be a pair of a unital nuclear \Cs -algebra and a non-degenerate state on $A$.
Let $\psi$ be a pure state on the matrix algebra $\mathbb{M}_n$ $(n \geq 2)$.
Then the reduced free product
$(A, \varphi)\ast (\mathbb{M}_n, \psi)$ is nuclear.
\end{Thm}
\begin{Thm}[Dykema {\cite[Theorem 2]{Dy}}]\label{Thm:Dy}
Let $(A, \varphi)$ and $(B, \psi)$ be pairs of a unital \Cs -algebra and a non-degenerate state on it.
Assume that $B\neq \mathbb{C}$ and
the centralizer of $\varphi$ contains a diffuse abelian \Cs -subalgebra $D$
containing the unit of $A$.
Then the reduced free product $(A, \varphi)\ast (B, \psi)$ is simple.
\end{Thm}
A good aspect of these theorems is that we only need to force a condition on one of the states.
Thus we can apply these theorems at the same time in many situations. 

\subsection{Extensions of \Cs -algebras}
Here we recall basic facts and terminologies related to the extensions of \Cs -algebras.
We refer the reader to \cite[Sections 15, 17]{Bla} for the details.
Let $A$ be a unital separable \Cs -algebra,
$B$ be a separable stable (i.e., $B\cong B\otimes \mathbb{K}$) nuclear \Cs -algebra.
Let $$0\rightarrow B\rightarrow C\rightarrow A\rightarrow 0$$ be an essential extension of $A$ by $B$.
Here essential means that the ideal $B$ of $C$ is essential (i.e., $cB=0$ implies $c=0$ for $c\in C$).

Let $\sigma\colon A \rightarrow Q(B):=M(B)/B$ be the Busby invariant of the above extension.
Here $M(B)$ denotes the multiplier algebra of $B$.
As usual, we identify an extension with its Busby invariant.
To define the addition of two extensions, we fix an isomorphism
$B\cong B\otimes \mathbb{K}$.
(Note that
up to canonical identifications, the choice of the isomorphism does not affect
to the following definitions.)
Then any isomorphism $\mathbb{K}\cong \mathbb{M}_2(\mathbb{K})$
induces the isomorphism $Q(B)\cong \mathbb{M}_2(Q(B))$.
Using this isomorphism, we can identify the direct sum of two Busby invariants
with a Busby invariant. This is independent on the choice of the isomorphism
$\mathbb{K}\cong \mathbb{M}_2(\mathbb{K})$ up to strong equivalence (defined below).

An extension $\sigma$ is said to be trivial (resp.~strongly unital trivial) if it has a $\ast$-homomorphism (resp.~unital $\ast$-homomorphism) lifting
$\tilde{\sigma}\colon A\rightarrow M(B)$.
Two extensions $\sigma_1$ and $\sigma_2$ are said to be strongly equivalent
if there is a unitary element $u$ in $M(B)$ satisfying
${\rm ad}(\pi(u))\circ \sigma_1=\sigma_2$.
An extension $\sigma$ is said to be absorbing (resp.~unital absorbing)
if for any trivial extension (resp.~strongly unital trivial extension) $\tau$,
$\sigma\oplus \tau$ is strongly equivalent to $\sigma$.
On the class of extensions of $A$ by $B$,
we define an equivalence relation as follows.
Two extensions $\sigma_1$ and $\sigma_2$ are equivalent
if there are trivial extensions $\tau_1$ and $\tau_2$
such that the direct sums $\sigma_i\oplus \tau_i$ are
strongly equivalent.
The quotient ${\rm Ext}(A, B)$ of the class of all extensions by this equivalence relation
naturally becomes an abelian semigroup.

Kasparov showed that there exists a unital absorbing strongly unital trivial extension $\tau$ of $A$ by
$B$ \cite[Theorem 6]{Kas}.
Therefore any $[\sigma]\in {\rm Ext}(A, B)$ has a unital absorbing representative.
Moreover, if $[\sigma]$ contains a unital extension,
then $[\sigma]$ has a unital absorbing unital representative.
Note that an element $[\sigma]\in {\rm Ext}(A, B)$ contains a unital extension
if and only if $[\sigma(1)]_0=0$ in $K_0(Q(B))$.

A theorem of Kasparov \cite[Theorem 2]{Kas} shows that for a unital absorbing extension $\sigma$, the direct sum $\sigma\oplus 0$
is an absorbing extension.
Thus, by the same reason as above, any element of ${\rm Ext}(A, B)$ has
an absorbing representative. By definition, such a representative is unique up to strongly equivalence.

It follows from \cite[Theorem 6]{Kas} that for any unital \Cs -subalgebra $C\subset A$,
the restriction of the absorbing (resp.~ unital absorbing) extension to $C$ again has the same property.

Let ${\rm Ext}(A, B)^{-1}$ be the subsemigroup of ${\rm Ext}(A, B)$
consisting of invertible elements.
Then there is a natural group isomorphism between ${\rm Ext}(A, B)^{-1}$ and 
$KK^1(A, B)$ \cite[Corollary 18.5.4]{Bla}.
Note that thanks to Kasparov's Stinespring type theorem \cite{Kas},
an extension $\sigma \colon A\rightarrow Q(B)$ is invertible in ${\rm Ext}(A, B)$ if and only if
it has a completely positive lifting $\tilde{\sigma} \colon A \rightarrow M(B)$.
See Section 15.7 in \cite{Bla} for details.

\section{Proof of Theorem \ref{Thmint:Main}}\label{Sec:proof}
Let $\Gamma$ be a countable exact group.
Take an amenable action $\Gamma\curvearrowright X$ on a compact metrizable space.
Define $A_n:=C(\prod_{k=n}^{\infty} X)\rtimes \Gamma$ for each $n\in \mathbb{N}$.
Here the action $\Gamma\curvearrowright \prod_{k=n}^{\infty} X$ is given by the diagonal action.
We regard $A_{n+1}$ as a \Cs -subalgebra of $A_n$ in the canonical way.
Since the $\Gamma$-space $\prod_{k=n}^{\infty} X$ is metrizable and amenable, each $A_n$ is separable and nuclear.
Put $A:=\bigcap_{n=1}^\infty A_n$.
We will show that
$A$ is isomorphic to an intermediate \Cs -algebra between ${\rm C}_{r}^\ast(\Gamma)$ and ${\rm C}^\ast_u(\Gamma)\cap L(\Gamma)$.
To see this, take an arbitrary point $x\in \prod_{k=1}^{\infty} X$
and define $\rho\colon C(\prod_{k=1}^{\infty} X)\rightarrow \ell^{\infty}(\Gamma)$ by
$\rho(f)(s):=f(s.x)$ for $f\in C(\prod_{k=1}^{\infty} X)$ and $s\in \Gamma$.
Then $\rho$ is a $\Gamma$-equivariant $\ast$-homomorphism.
Hence it induces a $\ast$-homomorphism $\tilde{\rho}\colon A_1 \rightarrow \ell^{\infty}(\Gamma)\rtimes \Gamma$.
Note that $E_g(a)\in \bigcap_{n=1}^\infty C(\prod_{k=n}^{\infty} X)$ for all $a\in A$ and $g\in \Gamma$.
We claim that the equality $\bigcap_{n=1}^\infty C(\prod_{k=n}^{\infty} X)=\mathbb{C}$ holds.
Indeed, any function contained in $\bigcap_{n=1}^\infty C(\prod_{k=n}^{\infty} X)$ is independent
on each coordinates. Such a continuous function must be constant. This proves the claim.
From this, it follows that $\tilde{\rho}$ is injective on $A$ and that the image $\tilde{\rho}(A)$ is contained in ${\rm C}^\ast_u(\Gamma)\cap L(\Gamma)$.
Thus $A$ is isomorphic to the desired \Cs -algebra.

Next we show that there is a compatible family of multiplicative conditional expectations
$(E_n\colon A_1\rightarrow A_n)_{n=1}^\infty$.
Define $E_n\colon A_1 \rightarrow A_n$ to be the $\ast$-homomorphism induced from the $\Gamma$-equivariant $\ast$-homomorphism
$$E_n\colon C(\prod_{k=1}^{\infty} X)\rightarrow C(\prod_{k=n}^{\infty} X)$$
defined by
$$E_n(f)(x_n, x_{n+1}, x_{n+2}, \ldots):=f(x_n, \ldots, x_n, x_{n+1}, x_{n+2}, \ldots),$$
where, in the right hand side, $x_n$ is iterated $n$ times.
Then it is not difficult to check that they satisfy the desired conditions.

To make terms isomorphic to the Cuntz algebra $\mathcal{O}_2$,
we first make terms simple.
To do this, take a faithful state $\nu$ on $A_1$.
Take a compact metric space $Y$ consisting at least two points and a faithful measure $\mu$ on $Y$.
On $\left(\bigotimes_{k=1}^\infty C(Y)\right)\otimes A_1$,
define a faithful state $\varphi$ by $\varphi:=\left(\bigotimes_{k=1}^\infty \mu \right)\otimes \nu$.
Then define a faithful state $\varphi_n$ on $B_n:=\left(\bigotimes_{k=n}^\infty C(Y)\right)\otimes A_n$
to be the restriction of $\varphi$.
Now take a pure state $\psi$ on $\mathbb{M}_2$
and put $C_n:= (B_n, \varphi_n)\ast\left( \bigfp_{k=n}^{\infty} (\mathbb{M}_2, \psi) \right)$.
Then by Theorem \ref{Thm:Dy},
each $C_n$ is simple.
Moreover, since $C_n$ is the increasing union of finite free products $\left((B_n, \varphi_n)\ast\left( \bigfp_{k=n}^m (\mathbb{M}_2, \psi) \right)\right)_{m=n}^\infty$,
each $C_n$ is nuclear by Theorem \ref{Thm:DS}.
For each $n\in \mathbb{N}$, by Theorem 4.8.5 in \cite{BO}, we have a conditional expectation
from $C_1$ onto $(B_1, \varphi)\ast\left( \bigfp_{k=1}^n (\mathbb{M}_2, \psi) \right)$ which maps $C_{n+1}$ onto $B_{n+1}$.
This proves the equalities
$$\bigcap_{n=1}^\infty C_n = \bigcap_{n=1}^\infty B_n=\bigcap_{n=1}^\infty A_n=A.$$

Finally, to make terms isomorphic to $\mathcal{O}_2$,
we apply Kirchberg--Phillips's $\mathcal{O}_2$-absorption theorem \cite{KP}.
We define a new sequence $(D_n)_{n=1}^\infty$
by $D_n:= C_n\otimes \left( \bigotimes_{k=n}^{\infty}\mathcal{O}_2 \right).$
Then each $D_n$ is isomorphic to $\mathcal{O}_2$.
Take a state $\xi$ on $\mathcal{O}_2$.
For each $n\in \mathbb{N}$, set
$E_n:= \id_{C_1} \otimes \left( \bigotimes _{k=1} ^n \id_{\mathcal{O}_2}\right) \otimes \left( \bigotimes _{k=n+1}^\infty \xi \right).$
Then $E_n$ maps $D_m$ onto $C_m$ for $m>n$ and
the sequence $(E_n)_{n=1}^\infty$ converges to the identity in the pointwise norm topology.
Therefore we have
$$\bigcap_{n=1}^\infty D_n =\bigcap_{n=1}^\infty C_n = A.$$
\qed
\begin{Rem}
There is an isomorphism between the decreasing intersection $A=\bigcap_{n\in \mathbb{N}} \left(C(\prod_{k=n}^{\infty} X)\rtimes \Gamma\right)$
and the \Cs -algebra
$$B=\{ b\in C(X)\rtimes \Gamma: E_g(b)\in \mathbb{C}{\rm\ for\ all\ }g\in \Gamma\}$$
that preserves the reduced group \Cs -algebra.
To see this, consider the quotient map $\pi \colon C(\prod_{k=1}^{\infty} X)\rtimes \Gamma \rightarrow C(X)\rtimes \Gamma$ induced from the diagonal embedding $X\rightarrow \prod_{k=1} ^\infty X$.
Note that $\pi$ is compatible with the coefficient maps and preserves the reduced group \Cs -algebra.
Since all elements of $A$ take the coefficients in $\mathbb{C}$,
we have the injectivity of $\pi$ on $A$.
To see the equality $\pi(A)=B$, for each $n\in \mathbb{N}$, consider the embedding
$\rho_n \colon C(X)\rtimes \Gamma \rightarrow C(\prod_{k=n}^{\infty} X)\rtimes \Gamma$ induced from the quotient map from $\prod_{k=1}^{\infty} X$ onto the $n$th product component.
Then, by comparing the coefficients, we have $\rho_n|_B=\rho_m|_B$ for all $n, m\in \mathbb{N}$.
Hence $\rho_1(B)=\bigcap_{n=1}^\infty \rho_n(B)\subset A$.
By comparing the coefficients, we further have $B=\pi(\rho_1(B))\subset \pi(A)$.
The converse inclusion is obvious by the definition of $B$.

Therefore, the question either the equation
$$\bigcap_{n\in \mathbb{N}}\left( C(\prod_{k=n}^{\infty} X)\rtimes \Gamma\right)= {\rm C}^\ast_r (\Gamma)$$ holds or not
seems difficult when the group $\Gamma$ does not have the AP.
Indeed, if the equation holds for every compact metrizable $\Gamma$-space $X$
(when $\Gamma$ is exact, we only need to consider the amenable one),
then $\Gamma$ has the ITAP.
However, we do not know either a given group has the ITAP or not for groups without the AP.
\end{Rem}
Now we can prove Corollary \ref{Corint:Main}.
\begin{proof}[Proof of Corollary \ref{Corint:Main}]
We apply Theorem \ref{Thmint:Main} to $\Gamma:={\rm SL}(3, \mathbb{Z})$.
(See \cite[Section 5.4]{BO} for the exactness of $\Gamma$.)
This gives an intermediate \Cs -algebra $A$ between ${\rm C}^\ast_{r}(\Gamma)$ and $L(\Gamma)\cap {\rm C}^\ast_u(\Gamma)$
satisfying the conditions in Theorem \ref{Thmint:Main}.
We show that $A$ does not have the OAP and the local lifting property.
Since $\Gamma$ does not have the AP \cite{LS},
Proposition \ref{Prop:AP} yields that $A$ does not have the OAP.

Next take a subgroup $\Lambda$ of $\Gamma$ isomorphic
to ${\rm SL}(2, \mathbb{Z})$.
Denote by $p$ the orthogonal projection from $\ell^2(\Gamma)$
onto the subspace $\ell^2(\Lambda)$.
Then the compression by $p$
gives a conditional expectation
$$E_{\Lambda}^\Gamma\colon {\rm C}^\ast_u(\Gamma)\rightarrow {\rm C}^\ast_u(\Lambda).$$
It is clear from the definition that
$E_{\Lambda}^\Gamma$ maps $L(\Gamma)\cap {\rm C}^\ast_u(\Gamma)$
onto $L(\Lambda)\cap {\rm C}^\ast_u(\Lambda)$.
Since $\Lambda$ has the AP \cite[Corollary 12.3.5]{BO}, we obtain the conditional expectation
$$\Phi\colon A\rightarrow {\rm C}^\ast_{r}(\Lambda).$$
Since ${\rm C}^\ast_{r}(\Lambda)$ does not have the local lifting property \cite[Corollary 3.7.12]{BO},
neither does $A$.
\end{proof}
\subsection*{Other examples}
We end  this section by giving few more examples satisfying the conditions in Theorem \ref{Thmint:Main}.
\begin{Prop}\label{Prop:NC}
Let $A$ be a unital separable nuclear \Cs -algebra, $\Gamma$ be a group with the AP.
Then for any action of $\Gamma$ on $A$,
the reduced crossed product $A\rtimes \Gamma$ satisfies the conditions
mentioned in Theorem \ref{Thmint:Main}.
\end{Prop}
Let $A$ be a unital \Cs -algebra. Let $\Gamma$ be a group and $S$ be a $\Gamma$-set.
Consider the reduced crossed product 
$A^{\otimes S}\rtimes \Gamma$
where $\Gamma$ acts on $A^{\otimes S}$
by the shift of tensor components.
We say it the generalized wreath product
of $A$ with respect to $S$ and denote it by $A\wr_S \Gamma$. We assume that $\Gamma$ and $S$ are countable.
\begin{Prop}\label{Prop:Perm}
The class of unital $\Cs$-algebras with the SOAP satisfying the conditions in Theorem \ref{Thmint:Main} is closed under taking the following operations.
\begin{enumerate}[\upshape (1)]
\item
The generalized wreath product with respect to any $\Gamma$-set with $\Gamma$ the AP.
\item Countable minimal tensor products.
\end{enumerate}
\end{Prop}
We note that the second condition in Theorem \ref{Thmint:Main} implies the first one
by the proof of Theorem \ref{Thmint:Main}.
Therefore to prove these propositions, we only need to check the second condition in Theorem \ref{Thmint:Main}.
To prove Propositions \ref{Prop:NC} and \ref{Prop:Perm}, we need the following proposition.
The idea of the proof is essentially contained in \cite{Zac}.
\begin{Prop}\label{Prop:AP2}
Let $\Gamma$ be a group with the AP.
Let $A$ be a $\Gamma$-${\rm C}^\ast$-algebra and let $X$ be a closed subspace of $A$.
Assume that an element $x\in A\rtimes \Gamma$ satisfies
$E_g(x) \in X$ for all $g\in \Gamma$.
Then $x$ is contained in the closed subspace $$X\rtimes \Gamma:=\overline{\rm span}\{xu_g: x\in X, g\in \Gamma\}.$$
Conversely, if the above implication always holds for any $\Gamma$-\Cs -algebra and its closed subspace,
then the group $\Gamma$ has the AP.
\end{Prop}

\begin{proof}
Since $\Gamma$ has the AP, there is a net
$(\varphi_i)_{i \in I}$ of finitely supported functions on $\Gamma$ satisfying the condition in Definition \ref{Def:AP}.
For $i \in I$, define the linear map
$\Phi_i \colon A\rtimes\Gamma\rightarrow A\rtimes \Gamma$
by $\Phi_i (y):=\sum_{g\in \Gamma} \varphi_i(g)E_g(y)u_g$.
We claim that the net $(\Phi_i)_{i \in I}$ converges to the identity map
in the pointwise norm topology.
To show this, consider the embedding
$\iota \colon A\rtimes\Gamma \rightarrow (A\rtimes\Gamma)\otimes {\rm C}_{r}^\ast(\Gamma)$
induced from the maps $a\in A\mapsto a\otimes 1$ and
$u_g\in \Gamma \mapsto u_g \otimes \lambda_g$.
(This indeed defines an embedding by Fell's absorption principle \cite[Prop.4.1.7]{BO}.)
Then the composite $\iota \circ \Phi_i$
coincides with the composite $({\rm id}_{A \rtimes \Gamma} \otimes m_{\varphi_i})\circ \iota$.
This proves the convergence condition.
Now let $x$ be as stated.
Then for any $i\in I$,
we have $\Phi_i(x) \in X\rtimes\Gamma$.
Since the net $(\Phi_i(x))_{i \in I}$ converges in norm to $x$,
we have $x \in X\rtimes \Gamma$.

To show the converse, for any \Cs -algebra $A$,
consider the trivial $\Gamma$-action on $A$.
Then we have an isomorphism $\varphi\colon A\rtimes \Gamma \cong {\rm C}^\ast_{r}(\Gamma) \otimes A$
satisfying $\varphi(au_g)=\lambda_g \otimes a$ for $a\in A$ and $g\in \Gamma$.
For any closed subspace $X$ of $A$, the isomorphism $\varphi$
maps $X\rtimes \Gamma$ onto ${\rm C}^\ast_{r}(\Gamma)\otimes X$.
Now take an element $x\in F({\rm C}^\ast_{r}(\Gamma), A, X)$.
To show the AP of $\Gamma$, by Proposition \ref{Prop:AP}, it suffices to show $x\in {\rm C}^\ast_{r}(\Gamma)\otimes X$.
For any $g\in \Gamma$, let $\tau_g$ be the bounded linear functional on ${\rm C}^\ast_{r}(\Gamma)$
satisfying $\tau_g(\lambda_h):=\delta_{g, h}$ for $h\in \Gamma$.
Then by the choice of $x$, we have $(\tau_g\otimes \id_A)(x)\in X$ for all $g\in \Gamma$.
Since $E_g(\varphi^{-1}(x))=(\tau_g\otimes \id_A)(x)$,
we conclude that $\varphi^{-1}(x)\in X\rtimes \Gamma$.
Thus we have $x \in {\rm C}^\ast_{r}(\Gamma) \otimes X$ as desired.
\end{proof}
As a consequence, we obtain a permanence property of the SOAP and the OAP.
\begin{Cor}\label{Cor:AP}
The SOAP and the OAP are preserved under taking the reduced crossed product of a group with the AP.
\end{Cor}
\begin{proof}
We only give a proof for the SOAP.
Let $A$ be a $\Gamma$-\Cs -algebra with the SOAP.
Let $B$ be a \Cs -algebra and $X$ be its closed subspace.
To show the SOAP of $A\rtimes \Gamma$, it suffices to prove the inclusion $F(A\rtimes \Gamma, B, X)\subset (A\rtimes\Gamma)\otimes X$.
Let $x\in F(A\rtimes \Gamma, B, X)$.
Then $(E_g\otimes \id_B)(x) \in F(A, B, X)$ for all $g\in \Gamma$.
Since $A$ has the SOAP, we have $F(A, B, X)\subset A\otimes X$.
Then from Proposition \ref{Prop:AP2}, we conclude $x\in (A\rtimes \Gamma)\otimes X$.
Here we use the canonical identification of $(A\rtimes \Gamma) \otimes B$ with $(A\otimes B)\rtimes \Gamma$.
\end{proof}
\begin{Rem}
The similar proofs also show the {\rm W}$^\ast$-analogues of
Proposition \ref{Prop:AP2} and Corollary \ref{Cor:AP}.
We note that the {\rm W}$^\ast$-analogue of Corollary \ref{Cor:AP} is shown by
Haagerup and Kraus for locally compact groups with the AP \cite[Theorem 3.2]{HK}.
\end{Rem}

\begin{proof}[Proof of Proposition \ref{Prop:NC}]
Replace $C(\prod_{k=n}^{\infty}X)$ by $C(\prod_{k=n}^{\infty}X)\otimes A$ with the diagonal $\Gamma$-action
in the proof of Theorem \ref{Thmint:Main}.
Then as in the proof of Theorem \ref{Thmint:Main}, the decreasing sequence
$((C(\prod_{k=n}^{\infty}X)\otimes A)\rtimes \Gamma)_n$ admits a compatible family of multiplicative conditional expectations.
By Theorem \ref{Prop:ame}, each term $(C(\prod_{k=n}^{\infty}X)\otimes A)\rtimes \Gamma$ is nuclear.
Their intersection contains $A\rtimes \Gamma$,
and the coefficients of elements in the intersection are contained in $\bigcap_{n=1}^\infty (C(\prod_{k=n}^{\infty}X)\otimes A)=A$.
Here the last equality follows from a similar argument to that in the proof of Theorem \ref{Thmint:Main}.
Now by Proposition \ref{Prop:AP2}, we conclude the equality
\[\bigcap_{n=1}^\infty\left((C(\prod_{k=n}^{\infty}X)\otimes A)\rtimes \Gamma\right)=A\rtimes \Gamma.\]
\end{proof}
\begin{proof}[Proof of Proposition \ref{Prop:Perm}]
(1):
First take a decreasing sequence $(A_n)_{n=1}^{\infty}$ of separable nuclear \Cs -algebras
whose intersection is isomorphic to $A$ and that admits a compatible family
of multiplicative conditional expectations.
We will use $C(\prod_{k=n}^{\infty} X)\otimes A_n^{\otimes S}$ instead of $C(\prod_{k=n}^{\infty}X)$ in
the proof of Theorem \ref{Thmint:Main}.
To complete the proof, it is enough to show the equality
$$\bigcap_{n=1}^\infty \left(C(\prod_{k=n}^{\infty} X)\otimes \left( A_n^{\otimes S} \right)\right)=A^{\otimes S}.$$
Indeed, if the equality holds, then, thanks to Proposition \ref{Prop:AP2},
the decreasing sequence $((C(\prod_{k=n}^{\infty} X)\otimes A_n^{\otimes S})\rtimes \Gamma)_{n=1}^\infty$
gives the desired sequence.
The inclusion `$\supset$' is clear.
Since $\bigcap_{n=1}^\infty C(\prod_{k=n}^{\infty} X)=\mathbb{C}$, by a similar argument to that in the proof of Theorem \ref{Thmint:Main}, we have
\[\bigcap_{n=1}^\infty \left(C(\prod_{k=n}^{\infty} X)\otimes \left( A_n^{\otimes S} \right)\right)= \bigcap_{n=1}^\infty A_n^{\otimes S}.\]
Fix a state $\varphi$ on $A_1$.
For a subset $T$ of $S$,
consider the conditional expectation
\[E_T:=(\id_{A_1})^{\otimes T} \otimes \varphi^{\otimes (S\setminus T)}\colon A_1^{\otimes S} \rightarrow A_1^{\otimes T}.\]
The net $(E_{\mathfrak{F}})_{\mathfrak{F}}$, where $\mathfrak{F}$ runs over finite subsets of $S$, then
converges to the identity in the pointwise norm convergence.
Moreover, for any $n$ and $T$, we have
$E_T(A_n^{\otimes S})=A_n^{\otimes T}$.
Hence, to show the above equality, it suffices to show the following claim.
For any $m\in \mathbb{N}$, we have
$\bigcap_{n=1}^\infty A_n^{\otimes m} =A^{\otimes m}$.
We show the claim by induction.
The case $m=1$ is trivial.
Suppose the claim is true for $m$.
Then for any $k\in \mathbb{N}$, we have
\[\bigcap_{n=1}^\infty A_n^{\otimes (m+1)}\subset F(A_k^{\otimes m}, A_1, A) = A_k^{\otimes m} \otimes A.\]
Here the last equality follows from the nuclearity of $A_k^{\otimes m}$.
Then by the SOAP of $A$ and our inductive hypothesis,
we further have 
\[ \bigcap_{k=1}^\infty (A\otimes A_k^{\otimes m} ) \subset F(A, A_1, \bigcap_{k=1}^\infty A_k^{\otimes m})
=A^{\otimes (m+1)}.\]
This completes the proof.

(2): Let $A_n$ be a sequence of \Cs -algebras whose terms satisfy the conditions in Theorem \ref{Thmint:Main}
and have the SOAP.
For each $n$, take a decreasing sequence $(B_{n, m})_{m=1}^\infty$ of nuclear \Cs -algebras with a compatible family of multiplicative conditional expectations
whose intersection coincides with $A_n$.
For each $m\in \mathbb{N}$, set $B_m:=\bigotimes_{n=1}^\infty B_{n, m}$.
Then it is not hard to check that each $B_m$ is nuclear and
the decreasing sequence $(B_m)_{m=1}^\infty$ admits a compatible family of multiplicative conditional expectations.
Moreover, by a similar argument to that in the previous paragraph,
it can be shown the equality $\bigcap_{m=1}^\infty B_m= \bigotimes_{n=1}^\infty A_n$.
This completes the proof.
\end{proof}

\section{Hyperbolic group case}\label{Sec:gen}
In this section, we give a geometric construction of a decreasing sequence of Kirchberg algebras whose
decreasing intersection is isomorphic to the hyperbolic group \Cs -algebra.
We construct such a sequence inside the boundary algebra $C(\partial \Gamma)\rtimes \Gamma$.
To find such a sequence, we construct amenable quotients of the boundary space.
The proof does not depend on both the reduced free product theory and Kirchberg--Phillips's $\mathcal{O}_2$-absorption theorem.
We also use the sequence constructed in this section for the free group case
in the next two sections.

For the definition and basic properties of hyperbolic groups,
we refer the reader to \cite[Section 5.3]{BO} and \cite{GH}.
(For the reader who is only interested in the free group case,
we recommend to concentrate on that case. In this case, some arguments related to geodesic paths become
much simpler.)
Here we only recall a few facts.
(See 8.16, 8.21, 8.28, and 8.29 in \cite{GH}.)
For a torsion free element $t$ of a hyperbolic group $\Gamma$, the sequence $(t^n)_{n=1}^\infty$
is quasi-geodesic.
The boundary action of $t$ has exactly two fixed points.
They are the points represented by the quasi-geodesic paths $(t^n)_{n=1}^\infty$ and $(t^{-n})_{n=1}^\infty$.
We denote them by $t^{+\infty}$ and $t^{-\infty}$ respectively.
For any neighborhoods $U_{\pm}$ of $t^{\pm\infty}$,
there is $n\in \mathbb{N}$
such that for any $m\geq n$,
$t^m(\partial \Gamma\setminus U_-)\subset U_+$ holds.

For a metric space $(X, d)$ and its points $x, y, z\in X$,
we denote by $\langle y, z \rangle_x$ the Gromov product $(d(y, x)+ d(z, x) - d(y, z))/2$
of $y, z$ with respect to $x$.

We recall the following criteria for the Hausdorffness of a quotient space.
We left the proof to the reader.
\begin{Prop}
Let $X$ be a compact Hausdorff space.
Let $\mathcal{R}$ be an equivalence relation on $X$.
Assume that
the quotient map $\pi \colon X\rightarrow X/\mathcal{R}$ is closed.
Then the quotient space $X/\mathcal{R}$ is Hausdorff.
\end{Prop}
The next lemma guarantees the amenability of certain quotients of amenable dynamical systems.
We are grateful to Narutaka Ozawa for letting us know Lusin's theorem \cite[Theorem 5.8.11]{Sri}.
\begin{Lem}\label{Lem:amenable criteria}
Let $\Gamma$ be a group,
$X$ be an amenable compact metrizable $\Gamma$-space.
Let $\mathcal{R}$ be a $\Gamma$-invariant equivalence relation on $X$
such that the quotient space $X/\mathcal{R}$ is Hausdorff.
Assume that each equivalence class of $\mathcal{R}$ is finite.
Then $X/\mathcal{R}$ is again an amenable compact $\Gamma$-space.
\end{Lem}
To prove Lemma \ref{Lem:amenable criteria}, we need the following characterization of amenability due to Anantharaman-Delaroche \cite[Theorem 4.5]{Ana}.
See also \cite[Prop.5.2.1]{BO} for a generalized version.
\begin{Prop}\label{Prop:Ana}
Let $\alpha\colon \Gamma\curvearrowright X$ be an action of $\Gamma$ on a compact metrizable space $X$.
Then $\alpha$ is amenable if and only if there exists a net
$(\zeta_i\colon X\rightarrow {\rm Prob}(\Gamma))_{i\in I}$ of Borel maps
with the following condition.
$$\lim_{i\in I}\int_X \|g.\zeta_i(x)-\zeta_i(g.x)\|_1~d\mu=0 {\rm\ for\ all\ }\mu \in{\rm Prob}(X) {\rm\ and\ }g\in \Gamma.$$
Here ${\rm Prob}(X)$ denotes the set of all Borel probability measures on $X$.
\end{Prop}
\begin{proof}[Proof of Lemma \ref{Lem:amenable criteria}]
First note that $\mathcal{R}$ must be closed in $X\times X$,
otherwise the quotient space $X/\mathcal{R}$ cannot be Hausdorff.
Then, since each equivalence class is finite,
Lusin's theorem \cite[Theorem 5.8.11]{Sri} tells us that
$\mathcal{R}$ is presented as a countable disjoint union of graphs of Borel maps between Borel subsets of $X$.
Then it is not hard to check that for each $f\in C(X)$,
the function
$\tilde{f}$ on $X/\mathcal{R}$ defined by
$$\tilde{f}([x]):=\frac{1}{\sharp[x]}\sum_{y\in [x]}f(y)$$
is Borel.
By the same reason, the similar formula also defines the map $\Phi$
from $C(X, {\rm Prob}(\Gamma))$ to $\mathcal{B}(X/\mathcal{R}, {\rm Prob}(\Gamma))$.
Here $C(X, {\rm Prob}(\Gamma))$ denotes the set of all continuous maps
from $X$ into ${\rm Prob}(\Gamma)$ and
$\mathcal{B}(X/\mathcal{R}, {\rm Prob}(\Gamma))$ denotes the set
of all Borel maps from $X/\mathcal{R}$ into ${\rm Prob}(\Gamma)$.

Let $(\zeta_i\colon X\rightarrow {\rm Prob}(\Gamma))_{i\in I}$ be a net of continuous maps that
satisfies the condition in the definition of amenability for $\Gamma\curvearrowright X$.
Consider the net $(\Phi(\zeta_i))_{i\in I}$.
Then for any $g\in \Gamma$, $x\in X$, and $i\in I$,
we have
$$\|( g.\Phi(\zeta_i))([x])-\Phi(\zeta_i)(g.[x])\|_1\\
\leq \frac{1}{\sharp [x]}\sum_{y\in [x]}\|g.\zeta_i(y)-\zeta_i(g.y)\|_1.$$
Thus, for each $g\in \Gamma$, the norms $\|( g.\Phi(\zeta_i))([x])-\Phi(\zeta_i)(g.[x])\|_1$
converge to $0$ uniformly on $X/\mathcal{R}$ as $i$ tends to $\infty$.
In particular, the net
$(\Phi(\zeta_i))_{i\in I}$
satisfies the condition in Proposition \ref{Prop:Ana}.
\end{proof}
\begin{Lem}\label{Lem:Sep}
Let $\Gamma$ be a hyperbolic group.
Let  $T$ be a finite set of torsion free elements of $\Gamma$.
Then the set
$$\mathcal{R}_{T}:= \Delta_{\partial \Gamma}　\cup \left\{ (g.t^{+\infty}, g.t^{-\infty}): g　\in \Gamma, t \in T\cup T^{-1} \right\}$$
is a $\Gamma$-invariant equivalence relation on $\partial \Gamma$.
Moreover, the quotient space $\partial\Gamma/ \mathcal{R}_{T}$ is a Hausdorff space.
\end{Lem}
\begin{proof}
Clearly $\mathcal{R}_{T}$ is $\Gamma$-invariant.
Let $s$ and $t$ be torsion free elements of $\Gamma$.
Then the two sets $\{s^{\pm\infty}\}$ and $\{t^{\pm\infty}\}$
are either disjoint or the same \cite[8.30]{GH}.
Therefore the set $\mathcal{R}_{T}$ is an equivalence relation.
Note that this shows that each equivalence class of $\mathcal{R}_{T}$ contains at most two points.

For the Hausdorffness of the quotient space, it suffices to show
that the quotient map $\pi\colon \partial \Gamma\rightarrow \partial \Gamma/\mathcal{R}_{T}$
is closed.
Let $A$ be a closed subset of $\partial \Gamma$.
Then $\pi^{-1}(\pi(A))=A\cup B$,
where
$$B:=\left\{g.t^{-\infty}\in \partial \Gamma : g\in \Gamma, t\in T\cup T^{-1}, g.t^{+\infty}\in A\right\}.$$
To show the closedness of $\pi(A)$, which is equivalent to that of $\pi^{-1}(\pi(A))$,
it suffices to show that ${\rm cl}(B)\subset A\cup B$.
Fix a finite generating set $S$ of $\Gamma$ and denote by $|\cdot |$ and $d(\cdot, \cdot)$
the length function and the left invariant metric on $\Gamma$ determined by $S$ respectively.
Take $\delta > 0$ with the property that every geodesic triangle in $(\Gamma, d)$ is $\delta$-thin \cite[Proposition 5.3.4]{BO}.
Let $x\in {\rm cl}(B)$ and take a sequence $(g_n.t_n^{-\infty})_{n=1}^\infty$ in $B$ which converges
to $x$.
By passing to a subsequence, we may assume that there is $t\in T\cup T^{-1}$
with $t_n=t$ for all $n\in \mathbb{N}$.
Replace $g_n$ by $g_n t^{l(n)}$ for some $l(n) \in \mathbb{Z}$ for each $n\in \mathbb{N}$,
we may further assume $|g_n|\leq |g_n t^k|$ for all $k\in \mathbb{Z}$ and $n\in \mathbb{N}$.
If the sequence $(g_n)_{n=1}^\infty$ has a bounded subsequence,
then it has a constant subsequence. Hence we have $x\in B$.
Assume $|g_n|\rightarrow \infty$.
For each $k\in \mathbb{Z}$, take a geodesic path $[e, t^k]$ from $e$ to $t^k$.
Since $t$ is torsion free, the sequences $(t^n)_{n=1}^\infty$ and $(t^{-n})_{n=1}^\infty$ are quasi-geodesic.
Therefore, by \cite[Prop.5.3.5]{BO},
there is $D>0$ such that the Hausdorff distance between
$[e, t^k]$ and $(t^n)_{n=0}^{k}$ is less than $D$ for all $k\in \mathbb{Z}$.
Now consider a geodesic triangle $\Delta$ with the vertices
$\{e, g_n^{-1}, t^k\}$. Let $f$ denote the comparison tripod of $\Delta$ (see Section 5.3 of \cite{BO} for the definition.)
Let $u, v, w$ be (unique) points in $\Delta$ lying on the geodesic paths
$[e, g_n^{-1}], [g_n^{-1}, t^k], [t^k, e] \subset \Delta$ respectively that satisfy $f(u)=f(v)=f(w)$.
Put $l_1:=d(e, u)=d(w, e), l_2:=d(u, g_n^{-1})=d(g_n^{-1}, v)$, and $l_3:=d(v, t^k)=d(t^k, w)$.
Then, since $\Delta$ is $\delta$-thin, we have $l_2 + \delta \geq {\rm dist}(g_n^{-1}, [e, t^k])\geq |g_n|-D.$
Since $l_1 + l_2=|g_n|$, this implies $l_1 \leq D +\delta.$
Then since $l_1+ l_3 =| t^k|$, we further obtain $l_3 \geq |t^k| -D -\delta$.
Combining these inequalities, we have $|g_nt^k| = l_2+ l _3 \geq |t^k|+|g_n|-2(D+\delta)$.
This yields
\[\langle g_nt^{k}, g_nt^{-l} \rangle_e\geq |g_n|-2(D+\delta) {\rm\ for\ all\ } n, k, l \in \mathbb{N}.\]
Since both $(t^k)_{k=1}^\infty$ and $(t^{-k})_{k=1}^\infty$ are quasi-geodesic and the left multiplication action of $\Gamma$ on itself is isometric,
the paths $\{(g_nt^k)_{k=1}^\infty, (g_nt^{-k})_{k=1}^\infty: n\in \mathbb{N}\}$ are uniformly quasi-geodesic
(i.e., there are constants $C\geq 1$ and $r > 0$ such that all paths in the set are $(C, r)$-quasi-geodesic).
This with the above inequality shows that
the limits of $(g_n.t^{+\infty})_{n=1}^\infty$ and $(g_n.t^{-\infty})_{n=1}^\infty$ coincide.
(Cf. Lemmas 5.3.5, 5.3.8 in \cite{BO} and the definition of the topology on $\partial \Gamma$.)
Since $A$ is closed, we have $x\in A$ as required.
\end{proof}

For a subgroup $\Lambda$ of a hyperbolic group $\Gamma$, we define
the limit set $L_\Lambda$ of $\Lambda$ to be the closure
of the set $\{t^{+\infty}\in \partial\Gamma: t\in\Lambda {\rm\ torsion\ free}\}$ in $\partial \Gamma$.
Recall that any hyperbolic group
does not contain an infinite torsion subgroup \cite[8.36]{GH}.
Therefore the limit set $L_\Lambda$ is nonempty when $\Lambda$ is infinite.
Since the equality $(sts^{-1})^{+\infty}=s.t^{+\infty}$ holds for any torsion free element $t\in \Gamma$ and any $s\in \Gamma$, the limit set $L_\Lambda$ is $\Lambda$-invariant.
Hence $\Lambda$ acts on $L_\Lambda$ in the canonical way.
Next we give two lemmas on the action on the limit set,
which are familiar to specialists.

We recall a few terminology.
Let $\alpha \colon \Gamma \curvearrowright X$ be an action of a group $\Gamma$
on a compact Hausdorff space $X$.
The action $\alpha$ is said to be minimal if all $\Gamma$-orbits are dense in $X$.
The action $\alpha$ is said to be topologically free
if for any $g\in \Gamma\setminus \{e\}$,
the set of all points of $X$ fixed by $\alpha(g)$ has the trivial interior.
\begin{Lem}\label{Lem:top}
Let $\Lambda$ be an ICC subgroup of a hyperbolic group $\Gamma$.
Then the action $\varphi_\Lambda$ of $\Lambda$ on its limit set $L_\Lambda$ is amenable, minimal, and topologically free.
\end{Lem}
\begin{proof}
The amenability of the boundary action shows that
the action $\varphi_\Lambda$ is amenable.
Since $\Lambda$ is ICC, it is neither finite nor virtually cyclic.
Hence $\Lambda$ contains a free group of rank $2$ by Theorem 8.37 of \cite{GH}.
Hence there are two torsion free elements $s$ and $t$ of $\Lambda$
which do not have a common fixed point.
This shows that any $\Lambda$-invariant closed non-empty subset $X$ of $L_\Lambda$ contains
at least four points, i.e., $s^{\pm \infty}$ and $t^{\pm \infty}$.
Therefore, for any torsion free element $u$ of $\Lambda$,
we can find an element of $X$ which is not equal to $u^{- \infty}$.
This shows that $X$ contains the set $\{u^{+\infty}: u\in \Lambda$ torsion free$\}$, which is dense in $L_\Lambda$.
Therefore $X=L_\Lambda$ and we conclude the minimality of $\varphi_\Lambda$.

Assume now that $\varphi_\Lambda$ is not topologically free.
Take an element $g_1 \in \Lambda\setminus \{e\}$ such that
the set $F_{g_1}:=\{x\in L_\Lambda:g_1.x=x\}$ has a nontrivial interior.
Then note that $L_\Lambda$ does not have an isolated point, since $\varphi_\Lambda$ is minimal
and $L_\Lambda$ is infinite. (Here the infiniteness of $L_\Lambda$ follows from the amenability of $\varphi_\Lambda$
and the non-amenability of $\Lambda$.)
Since any torsion free element of $\Lambda$ has only two fixed points in $L_\Lambda$, the order of $g_1$ must be finite. 
Assume $F_{g_1}=L_\Lambda$.
This means that the kernel of $\varphi_\Lambda$ is nontrivial.
Since any torsion free element of $\Lambda$ acts on $L_\Lambda$ non-trivially,
the kernel is a nontrivial torsion subgroup.
Therefore it must be finite.
This contradicts to the ICC condition.
For a subgroup $G$ of $\Lambda$, we set
$F_G:=\bigcap_{g \in G}F_g$.
Note that for a subgroup $G$ of $\Lambda$ and $g\in\Lambda$,
we have $F_{gGg^{-1}}=gF_G$.
Set $G_1:=\langle g_1\rangle.$
Then ${\rm int}(F_{G_1}) = {\rm int}(F_{g_1})\neq \emptyset$.
We will show that there is $g_2\in\Lambda$ satisfying
$$\emptyset \neq g_2({\rm int}(F_{G_1}))\cap{\rm int}(F_{G_1})\subsetneq {\rm int}(F_{G_1}).$$
Indeed, if such $g_2$ does not exist,
then the family $\{g({\rm int}(F_{G_1})): g\in\Lambda\}$ makes an open covering
of $L_\Lambda$ whose members are mutually disjoint.
(Note that if $g\in \Lambda$ satisfies $ {\rm int}(F_{G_1})\subsetneq g({\rm int}(F_{G_1}))$,
then $g^{-1}$ satisfies the required condition.)
This forces that the subgroup
$$\Lambda_0:=\left\{g\in \Lambda: g({\rm int}(F_{G_1}))={\rm int}(F_{G_1})\right\}$$
has finite index in $\Lambda$.
Since $\Lambda$ is ICC,
the subgroup $G:=\langle gG_1g^{-1}: g\in \Lambda_0\rangle$ must be infinite.
Moreover, by definition, we have ${\rm int}(F_{G})={\rm int}(F_{G_1})\neq 0$.
Hence $G$ must be an infinite torsion subgroup, a contradiction.
Thus we can take $g_2\in \Lambda$ as above.
Set $G_2=\langle G_1, g_2G_1{g_2}^{-1}\rangle$.
Then we have
$\emptyset\neq{\rm int}(F_{G_2})\subsetneq {\rm int}(F_{G_1})$.
This shows that $G_2$ is still finite and is larger than $G_1$.
Continuing this argument inductively, we obtain
a strictly increasing sequence $(G_n)_{n=1}^\infty$ of finite subgroups of $\Lambda$.
Then the union $\bigcup_{n=1}^\infty G_n$ is an infinite torsion subgroup of $\Lambda$,
again a contradiction. 
\end{proof}
\begin{Rem}
Conversely, if $\Lambda$ is not ICC, then the action on the limit set $L_\Lambda$ is not faithful.
In this case, $\Lambda$ contains a finite index subgroup $\Lambda_0$ with the nontrivial center.
Since $L_{\Lambda_0}=L_\Lambda$, the center of $\Lambda_0$ acts on $L_\Lambda$ trivially.
\end{Rem}

\begin{Lem}\label{Lem:dense}
For $\Lambda$ as in Lemma {\rm \ref{Lem:top}}, the equivalence relation
$$\mathcal{R}:=\left( \bigcup_{t\in \Lambda, {\rm\ torsion\ free}} \mathcal{R}_{{\{t\}}} \right) \cap (L_\Lambda\times L_\Lambda)$$ on $L_{\Lambda}$
is dense in $L_\Lambda\times L_\Lambda$.
\end{Lem}
\begin{proof}
Let $s$ and $t$ be two torsion free elements in $\Lambda$ which do not have a common fixed point.
For any neighborhoods $U_{\pm}$ of $s^{\pm \infty}$ and neighborhoods $V_\pm$ of $t^{\pm \infty}$
with the properties $U_+\cap V_-=\emptyset$ and $U_-\cap V_+=\emptyset$,
take a natural number $N$ satisfying
$s^{N}(\partial\Gamma\setminus U_-)\subsetneq U_+$ and
$t^{N}(\partial\Gamma\setminus V_-)\subsetneq V_+$.
Then, for any $m\in \mathbb{N}$, we have $(s^Nt^N)^m(\partial\Gamma\setminus V_-)\subsetneq U_+$
and $(s^Nt^N)^{-m}(\partial\Gamma\setminus U_+)\subsetneq V_-$.
This shows that the element $s^Nt^N$ is torsion free, $(s^Nt^N)^{+\infty}\in {\rm cl}(U_+)$, and $(s^Nt^N)^{-\infty}\in {\rm cl}(V_-)$.
Thus ${\rm cl}(U_+)\times {\rm cl}(V_-)$ intersects with $\mathcal{R}$.
Since both $U_+$ and $V_-$ can be arbitrarily small, this proves
the density of $\mathcal{R}$.
\end{proof}
Recall that an action $\Gamma\curvearrowright X$ of a group
on a compact Hausdorff space is said to be a locally boundary action
if for any nonempty open set $U \subset X$,
there is an open set $V\subset U$ and an element $t\in \Gamma$
satisfying
${\rm cl}(t.V)\subsetneq V$ \cite[Definition 6]{LaS}.
\begin{Lem}\label{Lem:loc}
Let $\Lambda$ and $\Gamma$ be as in Lemma \ref{Lem:top}.
Let $T$ be a finite set of torsion free elements of $\Lambda$.
Then the action $\Lambda\curvearrowright L_\Lambda/(\mathcal{R}_{T}\cap (L_\Lambda\times L_\Lambda))$ is a locally boundary action.
\end{Lem}
\begin{proof}
Let $s$ be a torsion free element of $\Lambda$
whose fixed points are not equal to $g.t^{\pm\infty}$ for any $g\in \Lambda$ and $t\in T$.
Then $\pi(s^{+\infty})\neq \pi(s^{-\infty})$.
Hence, on the set $\pi(L_\Lambda\setminus\{s^{+\infty}\})$,
the sequence $(s^n.x)_{n=1}^\infty$ converges to $\pi(s^{+\infty})$ uniformly on compact subsets.
Thus for any neighborhood $U$ of $\pi(s^{+\infty})$ whose closure
does not contain $\pi(s^{-\infty})$,
there is $n\in \mathbb{N}$ such that
$s^n({\rm cl}(U))\subsetneq U$.
From the minimality of $\Lambda\curvearrowright L_{\Lambda}$,
now it is easy to conclude that the action is a locally boundary action.
\end{proof}
\begin{Thm}
Let $\Lambda$ be a subgroup of a hyperbolic group $\Gamma$.
Then there is a decreasing sequence of nuclear \Cs -subalgebras of $C(L_\Lambda)\rtimes \Lambda$
whose intersection is equal to ${\rm C}_{r}^\ast(\Lambda)$.
Moreover, if $\Lambda$ is ICC, then we can find such a sequence with
the terms Kirchberg algebras in the UCT class.
\end{Thm}
\begin{proof}
Let $(\mathfrak{F}_n)_{n=1}^{\infty}$ be an increasing sequence of finite subsets of torsion free elements of $\Lambda$
whose union contains all torsion free elements.
Define $\mathcal{R}_n:=\mathcal{R}_{\mathfrak{F}_n}\cap (L_\Lambda \times L_{\Lambda})$ for each $n$.
Note that for each $n$, the quotient space $L_\Lambda/ \mathcal{R}_n$ is Hausdorff by Lemma \ref{Lem:Sep}.
Put $A_n:=C(L_{\Lambda}/\mathcal{R}_n)\rtimes \Lambda$.
Then by Lemma \ref{Lem:amenable criteria}, each $A_n$ is nuclear.
Moreover, by Lemma \ref{Lem:dense}, we have $\bigcap_{n=1}^\infty C(L_{\Lambda}/\mathcal{R}_n)=\mathbb{C}$.
Since every hyperbolic group is weakly amenable \cite{Oz}, we have the equality
$$\bigcap_{n=1}^\infty A_n ={\rm C}^\ast_{r} (\Lambda).$$
When $\Lambda$ is ICC, a similar proof to that of Lemma \ref{Lem:top} shows the topological freeness of $\Lambda\curvearrowright L_\Lambda/\mathcal{R}_n$.
Now by Lemma \ref{Lem:loc} and
Theorem 9 of \cite{LaS}, each $A_n$ is a Kirchberg algebra.
Since each $A_n$ is the groupoid \Cs -algebra of an amenable $\acute{{\rm e}}$tale groupoid,
they are in the UCT class by Tu's theorem \cite{Tu}.
\end{proof}

\section{Extensions of ${\rm C}_{r}^\ast(\mathbb{F}_d)$ by nuclear \Cs -algebras}\label{Sec:NE}
In this section, we prove Theorem \ref{Thmint:NE}.

We first consider the case $d$ is finite. We deal with the case $d=\infty$ in the end of this section.
Denote by $S$ the set of all canonical generators of $\mathbb{F}_d$.
Denote by $|\cdot |$ the length function on $\mathbb{F}_d$ determined by $S$.
To prove Theorem \ref{Thmint:NE}, first we compute the $K$-theory of the crossed product
$C(\partial\mathbb{F}_d/\mathcal{R}_S)\rtimes \mathbb{F}_d$.

We always use the following standard picture of the Gromov boundary $\partial \mathbb{F}_d$.
$$\partial\mathbb{F}_d:=\left\{ (x_n)_{n=1}^\infty\in \prod_{n\in \mathbb{N}} S\sqcup S^{-1}: x_n\neq x_{n+1}^{-1} {\rm\ for\ all\ } n\in \mathbb{N} \right\}$$
equipped with the relative product topology.
For $w\in \mathbb{F}_d$, we denote by $p[w]$
the characteristic function of the clopen set 
$$\left\{(x_n)_{n=1}^\infty \in \partial \mathbb{F}_d: x_1\cdots x_{|w|} =w\right\}$$
and set $q[w]:=p[w]+p[w^{-1}].$
Throughout this section, we identify $C(\partial\mathbb{F}_d/\mathcal{R}_S)$
with the $\mathbb{F}_d$-\Cs-subalgebra of $C(\partial\mathbb{F}_d)$ in the canonical way.
Under this identification, it is not difficult to check that for $s\in S$,
$q[s]$ is contained in $C(\partial \mathbb{F}_d/ \mathcal{R}_S)$.
We denote the action $\mathbb{F}_d\curvearrowright C(\partial \mathbb{F}_d)$
by $wf$ for $w\in \mathbb{F}_d$ and $f\in C(\partial\mathbb{F}_d)$.
\begin{Lem}\label{Lem:A1}
The \Cs -algebra $C(\partial\mathbb{F}_d/\mathcal{R}_S)$ is generated by the set 
$$\mathcal{P}:=\{wq[s]: w\in \mathbb{F}_d, s\in S\}.$$
In particular, the space $\partial\mathbb{F}_d/\mathcal{R}_S$
is homeomorphic to the Cantor set.
\end{Lem}
\begin{proof}
By the Stone--Weierstrass theorem, it suffices to show that
the set $\mathcal{P}$ separates the points of $\partial \mathbb{F}_d/\mathcal{R}_S$.
Let $x=(x_n)_{n=1}^\infty$ and $y=(y_n)_{n=1}^\infty$ be two elements
in $\partial\mathbb{F}_d$ satisfying $(x, y)\not\in \mathcal{R}_S$.
If $x\not\in \{ws^{+\infty}: w\in \mathbb{F}_d, s\in S\sqcup S^{-1}\}$,
then take $n\in \mathbb{N}$ with $x_n \neq y_n$.
Let $m$ be the smallest integer greater than $n$
satisfying $x_{m}\neq x_n$ (which exists by assumption).
Then the projection $(x_1\cdots x_{m-1})(q[x_m])$ separates $x$ and $y$.
Next consider the case $x=zs^{+\infty}, y=wt^{+\infty}$, where $s, t\in S\sqcup S^{-1}$ and
$z$, $w$ are elements of $\mathbb{F}_d$ whose last alphabets are not equal to $s^{\pm 1}$, $t^{\pm 1}$, respectively.
Assume $|z|\geq |w|$.
Since the last alphabet of $z$ is neither $s$ nor $s^{-1}$, we have $(zq[s])\leq p[z]$ and $(zq[s])(x)=1$.
When $z\neq w$, the above inequality with the assumption $|z| \geq |w|$ shows $(zq[s])(y)=0$.
Otherwise we have $z=w$ and $s\neq t^{\pm 1}$,
which again yield $(zq[s])(y)=0$.
Therefore $\mathcal{P}$ satisfies the required condition.

The last assertion now follows from the following fact.
A topological space is homeomorphic to the Cantor set
if and only if it is compact, metrizable, totally disconnected, and does not have an isolated point.
\end{proof}
\begin{Lem}\label{Lem:A3}
The $K_0$-group of $C(\partial\mathbb{F}_d/\mathcal{R}_S)\rtimes \mathbb{F}_d$
is generated by $\{[q[s]]_0: s\in S\}$.
\end{Lem}
\begin{proof}
By Lemma \ref{Lem:A1} and the Pimsner--Voiculescu exact sequence \cite{PV},
the $K_0$-group is generated by the elements represented by a projection in $C(\partial\mathbb{F}_d/\mathcal{R}_S)$.
Let $r$ be a projection in $C(\partial\mathbb{F}_d/\mathcal{R}_S)$.
Then $r$ can be presented as a sum
$\sum_{w\in F} p[w]$,
where $F$ is a subset of $\mathbb{F}_d\setminus\{e\}$ whose elements have the same lengths.
Let $w$ be an element of $\mathbb{F}_d$ whose reduced form is $s_{1}^{n(1)}\cdots s_{k}^{n(k)}$,
where $s_i\in S\sqcup S^{-1}$, $n(i)\in \mathbb{N}$, and $s_i\neq s_{i+1}$ for all $i$.
We define
$\hat{w} \in \mathbb{F}_d$ by $s_{1}^{n(1)}\cdots s_{k-1}^{n(k-1)}s_{k}^{-n(k)}$.
We will show that $w\in F$ implies $\hat{w}\in F$.
Indeed, if $w\in F$, then $r(ws_{k}^{+\infty})=1$.
Hence we must have $r(ws_{k}^{-\infty})=1$.
This implies $\hat{w}\in F$ as desired.
Since $w\neq \hat{w}$ and $[p[w]+p[\hat{w}]]_0=[q[s_{k}^{n(k)}]]_0$,
it suffices to show that for $s\in S$ and $n\in \mathbb{N}$, the element $[q[s^n]]_0$
is contained in the subgroup generated by $[q[s]]_0, s\in S$.
This follows from the equations
$$q[s^2]=sq[s]+s^{-1}q[s]+q[s]-2$$
\begin{center}
and
\end{center}
$$q[s^k]=sq[s^{k-1}]+s^{-1}q[s^{k-1}]-q[s^{k-2}]$$
for $s\in S$ and $k>2$.
\end{proof}
We denote the triplet $(K_0, [1]_0, K_1)$ by $K_\ast$.
\begin{Thm}\label{Thm:K}
The
$K_\ast(C(\partial\mathbb{F}_d/\mathcal{R}_S)\rtimes \mathbb{F}_d)$
is isomorphic to $(\mathbb{Z}^d, (1, 1, \ldots, 1), \mathbb{Z}^d)$.
\end{Thm}
\begin{proof}
We first compute the pair $(K_0, [1]_0)$.
By Lemma \ref{Lem:A3}, it suffices to show the linear independence of
the family $([q[s]]_0)_{s \in S}$.
Let $$\eta\colon C(\partial\mathbb{F}_d, \mathbb{Z})^{\oplus S} \rightarrow C(\partial\mathbb{F}_d, \mathbb{Z})$$
be the additive map defined by
$(f_s)_{s\in S}\mapsto \sum_{s\in S} (f_s-sf_s)$ and denote by $\tau$ the restriction of $\eta$
to $C(\partial\mathbb{F}_d/\mathcal{R}_S, \mathbb{Z})^{\oplus S}$.
Then the Pimsner--Voiculescu exact sequence \cite{PV} shows that
the canonical map
$$C(\partial\mathbb{F}_d/\mathcal{R}_S, \mathbb{Z}) \rightarrow K_0(C(\partial\mathbb{F}_d/\mathcal{R}_S)\rtimes \mathbb{F}_d)$$
is surjective and its kernel is equal to ${\rm im}(\tau)$.
Hence it suffices to show that ${\rm im}(\tau)$ does not contain
a nontrivial linear combination of the projections $q[s], s\in S$.
The isomorphisms $\ker(\eta)\cong K_1(C(\partial\mathbb{F}_d)\rtimes \mathbb{F}_d)\cong \mathbb{Z}^d$ (see \cite{Cun, PV, Spi})
show that $\ker(\eta)=\{(f_s)_{s\in S}: {\rm each\ }f_s{\rm\ is\ constant}\}$.
Now let $r=\sum_{s\in S}n(s)q[s]$ be a nontrivial linear combination of $q[s]$'s.
If $\sum_{s\in S} n(s) \not\equiv 0 \mod (d-1)$, then
$r \not \in {\rm im}(\eta)$ by \cite{Cun, Spi}.
(See also the second paragraph of \cite[Section 4]{Suz} for the detail.)
If $\sum_{s\in S} n(s) = (d-1)m$ for some $m\in \mathbb{Z}$,
then $\sum_{s\in S}n(s)q[s]=\eta((g_{s})_{s \in S})$,
where $g_s:=(n(s)-m)p[s^{-1}]$ for $s\in S$.
Hence $\eta^{-1}(\{r\})=(g_s)_{s\in S}+ \ker(\eta)$,
which does not intersect with $C(\partial\mathbb{F}_d/\mathcal{R}_S, \mathbb{Z})^{\oplus S}$.
Thus we have $r\not\in {\rm im}(\tau)$ in either case.

The isomorphism of the $K_1$-group follows from the Pimsner--Voiculescu exact sequence \cite{PV}
and the equality $\ker(\tau)=\ker(\eta)$.
\end{proof}

\begin{proof}[Proof of Theorem \ref{Thmint:NE}: the case $d$ is finite]
Let $A$ be a stable separable nuclear \Cs -algebra.
Let
$$\iota\colon  {\rm C}_{r}^\ast(\mathbb{F}_d)\rightarrow  C(\partial\mathbb{F}_d/\mathcal{R}_S)\rtimes \mathbb{F}_d$$
be the inclusion map.
Then the proof of Theorem \ref{Thm:K} yields that the homomorphism
$\iota_{\ast, 0}$ has a left inverse and the homomorphism $\iota_{\ast, 1}$ is an isomorphism.
Consequently, the homomorphism
$${\rm Hom}(K_i(C(\partial\mathbb{F}_d/\mathcal{R}_S)\rtimes\mathbb{F}_d), K_{1-i}(A))\rightarrow {\rm Hom}(K_i({\rm C}_{r}^\ast(\mathbb{F}_d)), K_{1-i}(A))$$ induced from $\iota$ is surjective for $i=0, 1$.
Recall that both ${\rm C}_{r}^\ast(\mathbb{F}_d)$ and $C(\partial\mathbb{F}_d/\mathcal{R}_S)\rtimes\mathbb{F}_d$ satisfy the universal coefficient theorem \cite[Corollary 7.2]{RS}.
Since $K_i({\rm C}_{r}^\ast(\mathbb{F}_d))$ is a free $\mathbb{Z}$-module for $i=0, 1$,
the universal coefficient theorem \cite{RS} yields that the canonical homomorphism
$${\rm Ext}({\rm C}_{r}^\ast(\mathbb{F}_d), A)^{-1}\rightarrow \bigoplus_{i=0, 1}{\rm Hom}(K_i({\rm C}_{r}^\ast(\mathbb{F}_d)), K_{1-i}(A))$$
is an isomorphism.
Combining these facts, we see that the homomorphism
$$\iota^\ast \colon {\rm Ext}(C(\partial\mathbb{F}_d/\mathcal{R}_S)\rtimes \mathbb{F}_d, A)\rightarrow {\rm Ext}({\rm C}_{r}^\ast(\mathbb{F}_d), A)^{-1}$$
induced from $\iota$ is surjective.

Now let $B$ be the exact \Cs -algebra obtained by an extension
$\sigma$ of ${\rm C}_{r}^\ast(\mathbb{F}_d)$ by $A$ which is either absorbing or unital absorbing.
Since $A$ is nuclear and ${\rm C}_{r}^\ast(\mathbb{F}_d)$ is exact,
the Effros--Haagerup lifting theorem \cite[Theorem B and Prop.~5.5]{EH} shows that $[\sigma] \in {\rm Ext}({\rm C}_{r}^\ast(\mathbb{F}_d), A)$ is invertible in the semigroup ${\rm Ext}({\rm C}_{r}^\ast(\mathbb{F}_d), A)$.
Note that in either case,
the direct sum $\sigma\oplus 0$ is absorbing.
Thus, by the surjectivity of $\iota^\ast$ and the unicity of absorbing representative, the direct sum $\sigma\oplus 0$ extends to a $\ast$-homomorphism
$\varphi\colon C(\partial\mathbb{F}_d/\mathcal{R}_S)\rtimes \mathbb{F}_d \rightarrow \mathbb{M}_2(Q(A)).$
Then, since $\varphi(1)=\sigma(1)\oplus 0\leq 1\oplus 0$,
the map
$$\tilde{\sigma}\colon C(\partial\mathbb{F}_d/\mathcal{R}_S)\rtimes \mathbb{F}_d \ni x \mapsto \varphi(x)_{1, 1}\in Q(A)$$
defines a $\ast$-homomorphism which extends $\sigma$.

We next show that $B$
is realized as a decreasing intersection of separable nuclear \Cs -algebras.
Take a decreasing sequence $(A_n)_{n=1}^\infty$ of nuclear \Cs-subalgebras of $C(\partial\mathbb{F}_d/\mathcal{R}_S)\rtimes \mathbb{F}_d$ whose intersection is equal to ${\rm C}_{r}^\ast(\mathbb{F}_d)$.
Put $B_n:=\pi^{-1}(\tilde{\sigma}(A_n))$ for each $n$,
where $\pi \colon M(A)\rightarrow Q(A)$ denotes the quotient map.
Then, since nuclearity is preserved under taking the extension,
each $B_n$ is nuclear.
Moreover, we have
$$\bigcap_{n=1}^\infty B_n =\bigcap_{n=1}^\infty \pi ^{-1}(\tilde{\sigma}(A_n))=B.$$

For the unital case, the rest of the proof is similarly done to the proof of Theorem \ref{Thmint:Main}.
For the non-unital case, let $(B_n)_{n=1}^{\infty}$ be a decreasing sequence of separable nuclear
\Cs -algebras whose intersection is $B$. 
Denote by $1$ the unit of the unitization $\widetilde{B_1}$ of $B_1$.
Define \Cs -subalgebras $C_n$ of $\widetilde{{B_1}}\oplus \ell^{\infty}(\mathbb{N})$
by
$$C_n:={\mathrm C}^\ast(B_n, \{ 1\oplus p_k: k\in \mathbb{N}\}),$$
where $p_k$ is the characteristic function of the set $\{l\in \mathbb{N}:l\geq k\}$.
Set $D_n:=C_n\otimes \bigotimes_{k=n}^{\infty}C(X)$ for each $n$,
where $X$ is a compact metrizable space consisting at least two points.
Take a faithful state $\phi$ on $C_1$ and a faithful measure $\mu$ on $X$.
Then define a state $\varphi$ on $D_1$
by $\varphi:=\phi\otimes\bigotimes_{k=1}^{\infty}\mu$.
Now take a pure state $\psi$ on $\mathbb{M}_2$
and define
\[E_n:=q_n\left((D_n, \varphi|_{D_n})\ast \left( \bigfp_{k=n}^{\infty}(\mathbb{M}_2, \psi)\right)\right)q_n,\]
where $q_n:=(1\oplus p_n)\in D_n$.
Then, being as a corner of a simple unital separable nuclear \Cs -algebra,
each $E_n$ also has these properties.
Now put $F_n:=E_n\otimes\bigotimes_{k=n}^{\infty}\mathcal{O}_2$.
Then each $F_n$ is isomorphic to $\mathcal{O}_2$ \cite{KP}.
Now by a similar argument to that in the last paragraph of the proof of Theorem \ref{Thmint:Main},
we have $\bigcap_{n=1}^\infty F_n =B$.

Finally, when $A=\mathbb{K}$,
by Voiculescu's theorem \cite{Vo}, any essential unital extension is unital absorbing
and any essential non-unital extension is absorbing.
Moreover, since ${\rm C}_{r}^\ast(\mathbb{F}_d)$ is simple \cite{Po}, the only non-essential extension is the zero extension
${\rm C}_{r}^\ast(\mathbb{F}_d)\oplus \mathbb{K}$. In this case, the claim follows from the above argument.
\end{proof}
We remark that in the proofs of Theorems \ref{Thmint:Main} and \ref{Thmint:NE},
the following is implicitly proved. Here we record it as a proposition.

\begin{Prop}
Let $A$ be a $($possibly non-unital$)$ \Cs-algebra which is realized as a decreasing intersection of separable nuclear \Cs -algebras.
Then it is realized as a decreasing intersection of isomorphs of the Cuntz algebra $\mathcal{O}_2$.
\end{Prop}

Now consider the case $d=\infty$.
\begin{proof}[Proof of Theorem \ref{Thmint:NE}: the case $d=\infty$]
Let $\Lambda$ be the commutator subgroup of $\mathbb{F}_2$. 
Then $\Lambda$ is isomorphic to $\mathbb{F}_\infty$.
Therefore we only need to show the claim for $\Lambda$.
Let $S$ be the canonical generator of $\mathbb{F}_2$
and consider the restriction $\alpha$ of the action $\mathbb{F}_2\curvearrowright \partial \mathbb{F}_2/\mathcal{R}_S$
to $\Lambda$.
Let $$\iota\colon {\rm C}^\ast_{r}(\Lambda)\rightarrow C(\partial \mathbb{F}_2/\mathcal{R}_S)\rtimes \Lambda$$
denote the inclusion.
We will show that the induced homomorphism $\iota_\ast$ on the K-theory
is left invertible.
To show the claim for the $K_0$-group, consider the following inclusion map
$$\tilde{\iota} \colon{\rm C}^\ast_{r}(\Lambda)\rightarrow C(\partial \mathbb{F}_2/\mathcal{R}_S)\rtimes \mathbb{F}_2.$$
Then by Theorem \ref{Thm:K}, the homomorphism $\tilde{\iota}_{\ast, 0}$ is left invertible.
This proves the left invertibility of $\iota_{\ast, 0}$. 

To show the claim for the $K_1$-group, first take a free basis $A$ of $\Lambda\cong \mathbb{F}_\infty$.
Define the homomorphism
$$\eta\colon C(\partial \mathbb{F}_2/\mathcal{R}_S, \mathbb{Z})^{\oplus A} \rightarrow C(\partial \mathbb{F}_2/\mathcal{R}_S, \mathbb{Z})$$
by $\eta((f_a)_{a\in A}):=\sum_{a\in A} (f_a-af_a).$
Then by the Pimsner--Voiculescu six term exact sequence,
we obtain an isomorphism
$$K_1(C(\partial \mathbb{F}_2/\mathcal{R}_S)\rtimes \Lambda)\cong \ker (\eta)$$
which maps $[u_a]_1$ to $(\delta_{a, b}1)_{b\in A}$
for each $a\in A$.
Since the subgroup generated by $1$ is a direct summand of the group $C(\partial \mathbb{F}_2/\mathcal{R}_S, \mathbb{Z})$,
the homomorphism $\mathbb{Z}^{\oplus A}\rightarrow \ker (\eta)$
given by $\delta_a \mapsto(\delta_{a, b}1)_{b\in A}$ is left invertible.
Consequently, the homomorphism $\iota_{\ast, 1}$ is left invertible.
Now the rest of the proof is similarly done to the case $d$ is finite.
\end{proof}
\begin{Rem}
Let $\sigma$ be a Busby invariant as before.
Denote by $$\varphi_i\colon K_i({\rm C}^\ast_{r}(\mathbb{F}_d))\rightarrow K_{1-i}(A)$$ the homomorphisms
corresponding to $\sigma$.
Then the six-term exact sequence gives the formula of the $K$-groups of the extension $B$ as follows.
$$K_i(B)\cong {\rm coker}(\varphi_{1-i})\oplus \mathbb{Z}^{\oplus d_i} (i=0, 1)$$
where $d_i:={\rm rank}(\ker(\varphi_i))$.
Furthermore, when $B$ is unital, the unit element $[1]_0$ corresponds to the element
\[u=\left\{ \begin{array}{ll}
0 & {\rm if\ } d_0=0, \\
0\oplus (1, 0 ,\ldots, 0) & {\rm otherwise}.\\
\end{array} \right.\]
\end{Rem}

Next consider two triplets $(G_0^{(j)}, u^{(j)}, G_1^{(j)}), j=1, 2,$ 
where $G_i^{(j)}$ are countable abelian groups and $u^{(j)}\in G_0^{(j)}$.
Then, by Theorem 4.1 of \cite{Phi}, every homomorphism between them is implemented by
a unital $\ast$-homomorphism between unital Kirchberg algebras in the UCT class.
Combining this fact with our results in this section, we obtain the following consequence.
\begin{Cor}\label{Cor:KK}
For any countable free group $\mathbb{F}$, there is a unital embedding
of ${\rm C}^\ast_{r}(\mathbb{F})$
into a Kirchberg algebra which implements the $KK$-equivalence.
\end{Cor}

\section{Consequences to amenable minimal Cantor systems of free groups}\label{Sec:app}
Let $d$ be a natural number greater than $1$.
Again let $S$ denote the set of canonical generators of $\mathbb{F}_d$.
Let $\mathfrak{F}\subset S$ be a nonempty proper subset.
Then a similar proof to Theorem \ref{Thm:K} shows that the space $\partial \mathbb{F}_d/\mathcal{R}_{\mathfrak{F}}$ is
the Cantor set and $K_\ast(C(\partial\mathbb{F}_d/\mathcal{R}_{\mathfrak{F}})\rtimes\mathbb{F}_d)$
is isomorphic to $(\mathbb{Z}^d, (1, 0, \ldots, 0), \mathbb{Z}^d)$.
(The set
$\{ [1]_0, [p[s]]_0, [q[t]]_0: s\in S\setminus \mathfrak{F}, t\in \mathfrak{F}'\}$
is a basis of the $K_0$-group for any subset $\mathfrak{F}'$ of $\mathfrak{F}$ with the cardinality $\sharp\mathfrak{F}-1$.
We remark that the equality $(d-\sharp \mathfrak{F}-1)[1]_0=-\sum_{s\in \mathfrak{F}}[q[s]]_0$ holds
in the $K_0$-group.)

The classification theorem of Kirchberg and Phillips \cite{Kir, Phi} shows that
the crossed products of the dynamical systems $\varphi_{\mathfrak{F}}\colon \mathbb{F}_d\curvearrowright \partial\mathbb{F}_d/\mathcal{R}_{\mathfrak{F}}$ are mutually isomorphic for nonempty subsets $\mathfrak{F}$ of $S$.
Moreover, these crossed products are isomorphic to a
Cuntz--Krieger algebra \cite[Lemma 3.7]{MM}.

Recall that two minimal topologically free actions on the Cantor set
is continuously orbit equivalent if and only if their transformation groupoids are
isomorphic as $\acute{{\rm e}}$tale groupoids \cite[Def.5.4]{Suz}. See \cite{Mat} and \cite{Suz} for relevant topics.
We will show that they are not continuously orbit equivalent.
Hence we obtain examples of amenable minimal Cantor $\mathbb{F}_d$-systems
which are not continuously orbit equivalent but have isomorphic crossed products.

Let $\alpha\colon \Gamma\curvearrowright X$ be an action of a group on the Cantor set.
Recall that the topological full group $[[\alpha]]$ of $\alpha$
is the group consisting homeomorphisms $h$ on $X$
satisfying the following condition.
For any $x\in X$, there is a neighborhood $U$ of $x$ and $g\in \Gamma$
such that $h(y)=g.y$ for all $y\in U$.
\begin{Thm}\label{Thm:CE}
Let $\mathfrak{F}$ and $\mathfrak{F}'$ be nonempty subsets of $S$.
Then $\varphi_{\mathfrak{F}}$ and $\varphi_{\mathfrak{F}'}$ are continuously orbit equivalent
if and only if $\sharp \mathfrak{F}=\sharp \mathfrak{F}'$.
\end{Thm}
\begin{proof}
Let $\mathfrak{F}$ be a nonempty subset of $S$
and denote by $[y]$ the element of $\partial \mathbb{F}_d /\mathcal{R}_{\mathfrak{F}}$ represented by $y\in \partial \mathbb{F}_d$.
Set
$$X:=\left\{x\in \partial\mathbb{F}_d/\mathcal{R}_\mathfrak{F}: {\rm there\ is\ }g\in \mathbb{F}_d\setminus \{e\}{\rm\ with\ }g.x=x \right\}.$$
Notice that the definition of $X$ only depends on the structure of the
transformation groupoid
$\partial\mathbb{F}_d/\mathcal{R}_{\mathfrak{F}}\rtimes \mathbb{F}_d$.
It is clear from the definition that
$$X=\{[w^{+\infty}]: w\in \mathbb{F}_d\setminus\{e\}\}.$$ 
For each $x\in X$, consider the following condition.
\begin{itemize}
\item[$(\ast)$]There is $F \in [[\varphi_{\mathfrak{F}}]]$
such that $x$ is an isolated point of the set of fixed points of $F$,
and both $(F^{n})_{n=1}^\infty$ and  $(F^{-n})_{n=1}^\infty$
do not uniformly converge to the constant map $y\mapsto x$ on any neighborhood of $x$.
\end{itemize}
Then it is easy to check that
$$Y:=\left\{x\in X: x{\rm\ satisfies\ the\ condition\ }(\ast)\right\}=\left\{[gw^{+\infty}]: g\in \mathbb{F}_d, w\in \mathfrak{F}\right\}.$$
Now the cardinality of $\mathfrak{F}$
is recovered as the number of the $\mathbb{F}_d$-orbits in $Y$.
\end{proof}
\begin{Rem}
It follows from Matui's theorem \cite[Theorem 3.10]{Mat} and Theorem \ref{Thm:CE}
that the topological full groups of $\varphi_{\mathfrak{F}}$ and $\varphi_{\mathfrak{F}'}$ are
not isomorphic when $\sharp\mathfrak{F}\neq \sharp\mathfrak{F}'$.
\end{Rem}

Now consider a one-sided irreducible finite shift $\sigma_A$ with 
$K_\ast(\mathcal{O}_A)\cong (\mathbb{Z}^d, (1, 1,\ldots, 1), \mathbb{Z}^d)$.
(Such one exists \cite[Lemma 3.7]{MM} and is unique up to continuously orbit equivalence \cite[Theorem 3.6]{MM}.)
Then, thanks to the classification theorem of Kirchberg and Phillips \cite{Kir, Phi},
for each nonempty subset $\mathfrak{F}$ of $S$, the crossed product $C(\partial\mathbb{F}_d/\mathcal{R}_{\mathfrak{F}})\rtimes\mathbb{F}_d$ is isomorphic to the Cuntz--Krieger algebra $\mathcal{O}_A$.
Thus the Cartan subalgebra $C(\partial\mathbb{F}_d/\mathcal{R}_{\mathfrak{F}})\subset
C(\partial\mathbb{F}_d/\mathcal{R}_{\mathfrak{F}})\rtimes\mathbb{F}_d$
provides a Cartan subalgebra of $\mathcal{O}_A$ whose spectrum is the Cantor set.
On the other hand,
the transformation groupoid $G_A$ of $\sigma_A$ (see \cite[Section 2.2]{MM} for instance) does not admit
a point satisfying the condition $(\ast)$ stated in the proof of Theorem \ref{Thm:CE}.
Thus our Cartan subalgebras are not conjugate to the canonical one $C(X_A)\subset \mathcal{O}_A$ \cite[Prop. 4.13]{Ren} (see also \cite[Theorem 5.1]{Mat0}).

\subsection*{Acknowledgement}
The author would like to thank Hiroki Matui,
who turns the author's interest to amenable quotients of the boundary actions.
He also thanks Narutaka Ozawa
for helpful discussions on hyperbolic groups and approximation theory.
He also thanks Caleb Eckhardt who raised a question about
properties of the decreasing intersection of nuclear \Cs-algebras with conditional expectations.
Finally, he is grateful to the referee, whose comments and suggestions
improve the presentation of the paper.
He was supported by Research Fellow
of the JSPS (No.25-7810) and the Program of Leading Graduate Schools, MEXT, Japan.


\begin{thebibliography}{99}
\bibitem{Ana} C.~Anantharaman-Delaroche, {\it Syst\`emes dynamiques non commutatifs et moyennabilit\'e.} Math.\ Ann.\ {\bf 279} (1987), 297--315.
\bibitem{Bla}B.~Blackadar, {\it K-theory for operator algebras.} Cambridge University Press, Cambridge (1998).
\bibitem{BO}
N.~P.~Brown, N.~Ozawa,
 {\it \Cs -algebras and finite-dimensional approximations.} 
Graduate Studies in Mathematics 88. American Mathematical Society, Providence, RI, 2008. xvi+509 pp.
\bibitem{Cun} J.~Cuntz, {\it A class of \Cs -algebras and topological Markov chains II: reducible chains and the Ext-functor
for \Cs -algebras.} Invent.\ Math.\ {\bf 63} (1981), 25--40.
\bibitem{Dy} K.~Dykema, {\it Simplicity and the stable rank of some free product C*-algebras.}
Trans.\ Amer.\ Math.\ Soc.\ {\bf 351} (1999), 1--40.
\bibitem{EH} E.~G.~Effros, U.~Haagerup, {\it Lifling problems and local reflexivity for ${\rm C}^*$-algebras.}
Duke Math.~J.~{\bf 52} (1985), 103--128.
\bibitem{GH} E.~Ghys, P.~de la Harpe, {\it Sur les groupes hyperboliques d'apr\'es Mikhael Gromov.}
Progress in Math.\ 83, (1990), Birkh\"{a}user.
\bibitem{HK} U.~Haagerup, J.~Kraus, {\it Approximation properties for group \Cs -algebras and group
von Neumann algebras.} Trans.\ Amer.\ Math.\ Soc.\ {\bf 344} (1994), 667--699.
\bibitem{Kas} G.~Kasparov, {\it Hilbert \Cs -modules: Theorems of Stinespring and Voiculescu.} J. Operator Theory
{\bf 4} (1980), 133--150.
\bibitem{Kir} E.~Kirchberg, {\it The classification of purely infinite \Cs -algebras using Kasparov's
theory.} Preprint.
\bibitem{KP} E.~Kirchberg, N.~C.~Phillips, {\it Embedding of exact \Cs -algebras in the
Cuntz algebra $\mathcal{O}_2$.} J. reine angew.\ Math.\ {\bf 525} (2000), 17–-53.
\bibitem{LaS} M.~Laca, J.~Spielberg, {\it Purely infinite \Cs -algebras from boundary actions
of discrete groups.} J. reine angew.\ Math.\ {\bf 480} (1996), 125--139.
\bibitem{LS} V.~Lafforgue, M.~de la Salle, {\rm Noncommutative $L^p$-spaces without the completely bounded
approximation property.} Duke Math.~J. {\bf 160} (2011), 71--116.
\bibitem{MM} K. Matsumoto, H. Matui, {\it Continuous orbit equivalence of topological Markov shifts and
Cuntz--Krieger algebras.}
Kyoto J.~Math.~ {\bf 54} (2014), 863--877.
\bibitem{Mat0} H. Matui, {\it Homology and topological full groups of etale groupoids on totally disconnected spaces.}
Proc.\ London Math.\ Soc.\ {\bf 104} (2012), 27--56.
\bibitem{Mat} H.~Matui, {\it Topological full groups of one-sided shifts of finite type.}
J.\ reine angew.\ Math. {\bf 705} (2015), 35--84.
\bibitem{Oz2}N.~Ozawa, {\it Amenable actions and exactness for discrete groups.}
C. R. Acad.\ Sci.\ Paris Ser.\ I Math.\ {\bf 330} (2000), 691--695.
\bibitem{Oz} N.~Ozawa, {\it Weak amenability of hyperbolic groups.} Groups Geom.\ Dyn.\ {\bf 2} (2008), 271--280.
\bibitem{Phi} N.~C.~Phillips. {\it A classification theorem for nuclear purely infinite simple \Cs -algebras.} Doc.\ Math.\ {\bf 5} (2000), 49--114.
\bibitem{PV} M.~Pimsner, D. Voiculescu, {\it $K$-groups of reduced crossed products by free groups.} J. Operator Theory {\bf 8} (1982), 131--156
\bibitem{Po} R.~T.~Powers, {\it Simplicity of the \Cs -algebra associated with the free group on two generators.}
Duke Math.~J.~ {\bf 42} (1975). 151--156.
\bibitem{Ren} J.~Renault, {\it Cartan subalgebras in \Cs -algebras.}
Irish Math.\ Soc.\ Bull.\ {\bf 61} (2008). 29--63.
\bibitem{Roe} J.~Roe, {\it Lectures on coarse geometry.}
University Lecture Series, vol. 31, American Mathematical Society, Providence, RI, 2003.
\bibitem{RS} J. Rosenberg, C. Schochet, {\it The Kunn\"{e}th theorem and the universal coefficient theorem for Kasparov's generalized $K$-functor.} Duke Math.~J.~{\bf 55} (1987), no. 2, 431--474.
\bibitem{Spi} J.~Spielberg, {\it Free product groups, Cuntz--Krieger algebras, and covariant maps.} Internat.\ J. Math.\ {\bf 2} (1991), 457--476.
\bibitem{Sri} S.~M.~Srivastava, {\it A course on Borel Sets.}
Grad.\ Texts in Math.\ 180 (1998), Springer.
\bibitem{Suz} Y.~Suzuki, {\it Amenable minimal Cantor systems of free groups arising from diagonal actions.} To appear in J.\ reine angew.\ Math.\ arXiv:1312.7098.
\bibitem{Tu} J.-L. Tu, {\it La conjecture de Baum--Connes pour les feuilletages moyennables.} $K$-theory
{\bf 17} (1999), 215--264.
\bibitem{Vo} D.~Voiculescu, {\it A non-commutative Weyl--von Neumann theorem.} Rev.\ Roumaine
Math.\ Pures Appl. {\bf 21} (1976), no. 1, 97--113.
\bibitem{Zac} J.~Zacharias,
{\it On the invariant translation approximation property for discrete groups.}
Proc.\ Amer.\ Math.\ Soc.\ {\bf 134} (2006), 1909--1916.

\end{thebibliography}
\end{document}